\documentclass[11pt,twoside,reqno]{amsart}
\usepackage{amsmath, amsfonts, amsthm, amssymb,amscd, graphicx, amscd}

\usepackage{movie15}
\usepackage{cancel}
\usepackage{float,epsf}
\usepackage[english]{babel}
\usepackage{enumerate}
\usepackage{tikz}
\usepackage{hyperref}

\usepackage{srcltx}

\usepackage{geometry}
\usepackage{verbatim}
\usepackage{mathrsfs}

\graphicspath{ {11./images/} }

\newcommand{\R}{\mathbb{R}}

\newtheorem{theorem}{Theorem}[section]
\newtheorem{lemma}[theorem]{Lemma}

\newtheorem{proposition}[theorem]{Proposition}
\newtheorem{corollary}[theorem]{Corollary}
\newtheorem{remark}[theorem]{Remark}

\newtheorem{question}[theorem]{Question}
\newtheorem{conjecture}{Conjecture}
\newtheorem*{remark*}{Remark}

\numberwithin{equation}{section}
\numberwithin{figure}{section}

\newcommand{\field}[1]{\mathbb{#1}}
\newcommand{\C}{\field{C}}

\DeclareMathOperator{\sech}{sech}

\def\intave#1{\int_{#1}\hbox{\llap{$\raise2.3pt\hbox{\vrule
height.9pt width7pt}\phantom{\scriptstyle{#1}}\mkern-2mu$}}}

\begin{document}
\title{Rigidity results on Liouville equation}

\author{Alexandre Eremenko}
\address{Alexandre Eremenko, Purdue University, West Lafayette IN, USA.}
\email{eremenko@math.purdue.edu}

\author{Changfeng Gui}
\address{Changfeng Gui, The University of Texas at San Antonio, TX, USA.
}
\email{changfeng.gui@utsa.edu}

\author{Qinfeng Li}
\address{Qinfeng Li, Hunan University, Changsha, Hunan, China.
}
\email{liqinfeng1989@gmail.com}

\author{Lu Xu}
\address{Lu Xu, Hunan University, Changsha, Hunan, China.
}
\email{xulu@hnu.edu.cn}

\maketitle
\begin{abstract}
We give a complete classification of solutions bounded from above
of the Liouville equation
$$-\Delta u=e^{2u}\quad\mbox{in}\quad \R^2.$$
More generally, solutions in the class
$$N:=\{ u:\limsup_{z\to\infty} u(z)/\log|z|:=k(u)<\infty\}$$
are described.
As a corollary, we obtain five rigidity results.
First, $k(u)$ can take only a discrete set of values: either $k=-2$,
or $2k$ is a non-negative integer. Second, $u\to-\infty$ as $z\to\infty$, if and only if $u$
is radial about some point. Third, if $u$ is symmetric with respect to $x$
and $y$ axes and $u_x<0,\, u_y<0$ in the first quadrant, then $u$ is radially symmetric. Fourth, if $u$ is concave and bounded from above, then $u$ is one-dimensional. Fifth, if $u$ is bounded from above, and the diameter of $\R^2$ with the metric
$e^{2u}\delta$ is $\pi$, where $\delta$ is the Euclidean metric, then $u$ is either
radial about a point or one-dimensional.

In addition we also extend the concavity rigidity result on Liouville equation in higher dimensions.

\end{abstract}

\section{Introduction}
\subsection{Motivating questions}
This paper is motivated by four questions on rigidity of solutions to the following Liouville equation in the plane.
\begin{align}
 \label{liouvilleequation}
    -\Delta u=e^{2u},\quad \mbox{in $\mathbb{R}^2$.}
\end{align}
\begin{question}
\label{q1}
Let $u$ be a solution to \eqref{liouvilleequation}. If $\lim_{|z|\rightarrow \infty}u(z)=-\infty$, does it follow that $u$ is radial about a point?
\end{question}
\begin{question}
\label{q1'}
Let $u$ be a solution to \eqref{liouvilleequation} such
that $u$ is symmetric about both the $x$-axis and the $y$-axis, and
satisfies
$u_x<0$ and $u_y<0$ in the first quadrant. Must $u$ be radial?
\end{question}
\begin{question}
\label{q2}
If $u$ is a concave solution to \eqref{liouvilleequation}, are
all level sets of $u$ flat?
\end{question}
\begin{question}
\label{q3}
Let $u$ be a solution to \eqref{liouvilleequation} and $\delta$ be the Euclidean metric in $\mathbb{R}^2$. If $\mathbb{R}^2$ has diameter $\pi$ under the metric $e^{2u}\delta$, is it true that $u$ must either be radial about a point, or $u$ is one-dimensional?
\end{question}

Note that solutions which are radial about a point satisfy the condition
in Question \ref{q1}, and hence it is a natural question to
ask whether the converse statement is true. Unfortunately,
the moving plane method in Chen-Li
\cite{CL} cannot be implemented since it is not known how fast $u$ decays
to $-\infty$.

Question \ref{q1'} was proposed in order to understand the global bifurcation
related to mean field equations on flat tori. The difficulty is
that the assumptions of this question do not imply finiteness of the
area $\int_{\R^2} e^{2u}$ and thus the result of Chen-Li \cite{CL}
cannot be used.

Question \ref{q2} was proposed in \cite{LX2021} by the third and fourth author.
The motivation comes from connection between the Liouville equation and the De Giorgi conjecture on phase transitions. A one dimensional solution to \eqref{liouvilleequation} looks like $u(x,y)=\ln(\sech y)$ up to translation, scaling and rotation. Surprisingly, $$\frac{\partial u}{\partial y}=-\tanh y,$$ which up to scaling is the one-dimensional solution to the classical Allen-Cahn equation
       \begin{align*}
           \Delta u=u^3-u.
       \end{align*}
The De Giorgi conjecture says that if along a direction,
the directional derivative of a solution to the Allen-Cahn equation is positive,
then level sets of the solution must be flat when the dimension of
the space is less than or equal to $8$.
A natural analogous assumption in the case of Liouville equation
is the sign constancy of the Hessian, from which the question arises.
Motivated by this question, in \cite{LX2021}, we developed a stronger
constant rank theorem which is useful to study concave solutions,
but we still could not solve Question \ref{q2}.

Question \ref{q3} was proposed in \cite{GL} by the second and third author,
and is important in the field of metric geometry.
Solving \eqref{liouvilleequation} means finding a conformal
factor $e^{2u}$ such that for the metric $e^{2u}\delta$,
where $\delta$ is the standard metric in $\mathbb{R}^2$,
the Gaussian curvature is equal to $1$ everywhere.
Usually, this metric is incomplete, and hence classical comparison
results in Riemannian geometry cannot apply.
In \cite{GL}, it is shown that in general, the conformal diameter is
at least $\pi$, and thus it is interesting to find out
when the lower bound $\pi$ is achieved. This motivates Question \ref{q3}.

In this paper we solve questions \ref{q1}-\ref{q1'}, and give partial solutions of questions \ref{q2}-\ref{q3}.

\medskip

\subsection{Main results}
Before stating our results, we first recall the connection between the Liouville equation and meromorphic function theory.

The Liouville equation \eqref{liouvilleequation} in the complex plane $\C$ describes a metric of constant curvature $1$
which is conformal to the standard Euclidean metric. The length
element of this metric is $e^{u(z)}|dz|$. Since all metrics of constant curvature $1$ are locally the same,
we have a developing map $f:\mathbb{C}\to\bar{\mathbb{C}}$, where $\bar{\mathbb{C}}$ is the Riemannian
sphere with the standard metric whose length element is
$$\frac{2|dz|}{1+|z|^2},$$
and $f$ is a local isometry. So $f$ can be considered as a
{\em locally univalent} meromorphic\footnote{We reserve this name for functions
meromorphic in the complex plane $\C$.} function ($f'(z)\neq 0,\; z\in\C$ and
all poles are simple). The relation between $u$ and $f$ is
the following:
\begin{equation} \label{complexrepresentation}
u(z)=\log\frac{2|f'(z)|}{1+|f(z)|^2}.
\end{equation}
So each solution $u$ of \eqref{liouvilleequation} has this representation, and conversely
every locally univalent meromorphic function $f$ defines a solution
of \eqref{liouvilleequation}. Two meromorphic functions $f_1$ and $f_2$ define the same
$u$ if and only if $f_1=\phi\circ f_2$, where $\phi$ is a rotation
of the sphere:
$$\phi(z)=\frac{pz-\bar{q}}{qz+\overline{p}},\quad p,q \in \mathbb{C}, \quad |p|^2+|q|^2=1.$$

\begin{remark}
\label{ts}
If $u(z)$ is a solution to \eqref{liouvilleequation}, then for any $z_0\in \mathbb{C}$, $\lambda \in \mathbb{C}\setminus \{0\}$,\begin{align*}
    \ln |\lambda|+u(\lambda z+z_0)
\end{align*}
is also a solution, which corresponds to the developing map $f(\lambda z+z_0)$.
\end{remark}

\subsubsection{Results in two dimensions}

Instead of dealing with questions \ref{q1}-\ref{q3} case by case we
give a complete classification of solutions bounded from above, which
will imply a complete solution of questions \ref{q1}-\ref{q1'},
and solution of questions \ref{q2}-\ref{q3} in this class.

The classification result is as follows.

\begin{theorem}
\label{upb}
Let $u$ be a solution to the Liouville equation \eqref{liouvilleequation}. Then $u$ is bounded from above, if and only if either
\begin{align}
\label{cr}
    u(x,y)=\ln \left(\frac{2}{1+x^2+y^2}\right),
\end{align}or
\begin{align}
\label{u_t}
u(x,y)=\ln\left(\frac{2e^x}{1+t^2+2te^x \cos y+e^{2x}}\right)
\quad\mbox{for some}\quad t\geq 0,
\end{align} up to a transformation in Remark \ref{ts}.
\end{theorem}
As a corollary of Theorem \ref{upb}, we summarize our results on the motivating questions stated in the beginning of the paper.
\begin{theorem}
\label{answer}
Let $u$ be a solution to the Liouville equation \eqref{liouvilleequation}. Then we have the following:
\begin{enumerate}
\item $\lim_{|z|\rightarrow \infty} u(z)=-\infty$, if and only if $u$ is radial about a point.
\item If $u$ is symmetric with respect to both $x$-axis and $y$-axis and $u_x<0,\, u_y<0$ in the first quadrant, then $u$ is radially symmetric.
\item If $u$ is concave and bounded from above, then up to
a transformation in Remark \ref{ts}, $u(z)=\ln(\sech(y))$.
\item If $u$ is bounded from above, then under the metric $e^{2u}\delta$,
the diameter of $\mathbb{R}^2$ belongs to $[\pi,2\pi)$
and can take any value in this interval. Moreover,
the diameter is equal to $\pi$, if and only if either $u$
is radial about some point, or $u$ is a one-dimensional solution.
\end{enumerate}
\end{theorem}

\begin{remark}
By later Remark \ref{class1d}, $u$ is a one-dimensional
solution if and only if $u$ takes the form \eqref{u_t}
for $t=0$, up to the transformations in Remark \ref{ts},
\end{remark}

Except for the concavity problem where we have a pure PDE approach actually
working for all dimensions as we will mention later, it seems to be rather
difficult to prove other results in Theorem \ref{answer} by pure PDE techniques.
In fact, using representation \eqref{complexrepresentation}
and referring to meromorphic function theory, we can analyze a larger class
of solutions to \eqref{liouvilleequation}.

The class $N$ of solutions discussed here is defined by their behavior at infinity, $u(z)\leq O(\log|z|),\; z\to\infty.$
For a real number $k$ we define $N(k)$ to be the set of all $u\in N$
such that
$$\limsup_{z\to\infty}\frac{u(z)}{\log|z|}=k.$$

\begin{theorem}
\label{e1}
$N(k)$ is non-empty if and only if either
$k=-2$, or $2k$ is a non-negative integer.

Class $N(-2)$ consists of those solutions for which the function $f$
in \eqref{complexrepresentation} is a linear-fractional transformation. For all other solutions of
\eqref{liouvilleequation}, function $f$ is transcendental.

For $k\geq0$, $u\in N(k)$ if and only if $f=w_1/w_2$, where $w_1$ and $w_2$
are two linearly independent solutions of the complex differential equation
\begin{equation}\label{ODE}
w''+P(z)w=0,\quad z\in \mathbb{C},
\end{equation}
where $P$ is a polynomial of degree $d=2k$.
\vspace{.1in}

Moreover, if $u \in N(k)$ for $k \ge 0$, then there exists $c>0$ and $\theta_0 \in \mathbb{R}$ such that
\begin{equation}\label{asymptotics}
u(re^{i\theta})=-cr^{k+1}|\sin\left((k+1)(\theta-\theta_0)\right)|+o(r^{k+1}),\quad r\to\infty.
\end{equation}
Moreover, $u(z)\sim k\log |z|,\; |z|\to\infty,$
on some curves tending to $\infty$ with asymptotic directions $\theta$
satisfying $\sin\left(k(\theta-\theta_0)\right)=0$.
\end{theorem}

The main ingredient in proving Theorem \ref{e1} is Nevanlinna theory and asymptotic integration theory for the linear ODE \eqref{ODE}. Then with some more effort, we prove Theorem \ref{upb} and Theorem \ref{answer}.


\medskip

\subsubsection{Results in higher dimensions}
Now we state our concavity rigidity results on solutions to the Liouville equation in higher dimensions, which is actually more meaningful as counterpart of the De Giorgi conjecture.

 We consider
\begin{align}
    \label{nliouvilleequation}
-\Delta u =e^{2u}, \quad \mbox{in $\mathbb{R}^n$}.
\end{align}Even though \eqref{nliouvilleequation} has no geometric background when $n\ge 3$, it arises in the theory of gravitational equilibrium of polytropic stars, and classification of solutions to \eqref{nliouvilleequation} is crucial in the study of the corresponding problems in bounded domains, as pointed in \cite{Farina}. We remark that \eqref{nliouvilleequation} has been studied in \cite{Farina} and \cite{DF}. In particular, Dancer and Farina prove that  there are no finite Morse index solutions to \eqref{liouvilleequation} when $3\le n \le 9$, and when $n=2$, finite Morse index solutions must be radial about a point.

We would like to see whether concave solutions to \eqref{nliouvilleequation} are one-dimensional in certain dimensions. As mentioned, this is related to the De Giorgi conjecture. Recall that due to Ghoussoub-Gui \cite{GG98} and Ambrosio-Cabr\'e \cite{AC20},  De Giorgi's conjecture has been proved
in dimensions 2 and 3.  In dimensions 4 and 5 under an odd symmetry condition
\cite{GG03}, and in dimension 4 to 8 under a limit condition \cite{Savin}
(see also the work of Wang \cite{Wang2017}
where a different approach is given), the conjecture is also proven.
 When $n \ge 9$, a counterexample is given by Del Pino-Kowalczyk-Wei \cite{DKW}.

One can see that the De Giorgi conjecture is sensitive to dimensions. Surprisingly, in the Liouville equation case, we prove that with the additional assumption that a concave solution achieves its finite supremum, then in any dimension, such solution must be one-dimensional.

\begin{theorem}
\label{main}
Let $u$ be a concave solution to \eqref{nliouvilleequation}. If further we assume that $u$ has a local maximum at some point $x_0\in \mathbb{R}^n$, then $u=\ln (\sech x_n)$ up to transformation in Remark \ref{ts}.
\end{theorem}

We give some remarks on our proof. In higher dimensions, we can no longer use complex analysis tools. The key ingredient in obtaining Theorem \ref{main} is the constant rank theorem, which states that if $u$ is a convex solution to some elliptic equation under some conditions, then $D^2 u$ has constant rank. The proof also works in dimension 2.

For the constant rank theorem, we refer to Caffarelli-Friedman \cite{CF} for the case of semilinear equations in $\mathbb{R}^2$ and Korevaar-Lewis \cite{KL87} for the case of semilinear equations in $\mathbb{R}^n$. See also Bian-Guan \cite{BG}, Bian-Guan-Ma-Xu \cite{BGMX}, Caffarelli-Guan-Ma \cite{CGM}, Guan-Lin-Ma \cite{GLM}, Guan-Ma \cite{GM}, Guan-Ma-Zhou \cite{GMZ}, Guan-Xu \cite{GX}, Ma-Xu \cite{MX08}, etc, for the case of other nonlinear elliptic equations with applications.

If $u$ is a solution to \eqref{nliouvilleequation}, then $v:=-u$ satisfies
\begin{align*}
    \Delta v=G(v) \quad \mbox{in $\mathbb{R}^n$},
\end{align*}where $G(v)=e^{-2v}$. It is proved in \cite{KL87} that actually if $GG^{''}<2(G')^2$, then $D^2 v$ has constant rank $r$, and $v$ is a constant in $(n-r)$ coordinate directions. Hence the constant rank theorem can apply in our case.

In order to apply the constant rank theorem to illustrate that concave solutions are one-dimensional, we need to exclude the possibility of strictly concave solutions. With the local maximum condition, this situation can be ruled out, by studying asymptotic behavior of strictly concave solutions and then proving nonexistence of such solutions.

\medskip\

The organization of the paper is as follows. In section 2,
we recall the link between solutions of the Liouville equation
in $\mathbb{R}^2$ and
spherical derivatives of meromorphic functions,
and we classify solutions which are radial about a point or one-dimensional.
In section 3, we first prove Theorem \ref{e1}, and then we prove Theorem \ref{upb} and Theorem \ref{answer}. In section 4, we prove Theorem \ref{main}. In section 5, we give some further remarks.

\section{Liouville's formula and spherical derivative}
In this section, we recall the relation between solutions to the
 Liouville equation \eqref{liouvilleequation} and the
spherical derivatives of meromorphic functions.
Also, we classify all radial and one-dimensional solutions.

Given a solution $u$, the function $f$
in \eqref{complexrepresentation} is called the \textit{developing function}
or \textit{developing map} for $u$.
By \eqref{complexrepresentation}, the conformal factor
of the metric is
\begin{align}
    \label{sphericalderivative}
f^\#(z):=\frac{2|f'(z)|}{1+|f(z)|^2},
\end{align}which is called the \textit{spherical derivative} of $f$.

\begin{remark}
\label{sphericallength}
Recall that the length element of the spherical metric is $2|dz|/(1+|z|^2)$.
The spherical length of a curve $\gamma(t),\, a\le t\le b$ is given by
\begin{align}
    \label{sphlength}
\int_{\gamma}\frac{2|dz|}{1+|z|^2}=\int_a^b \frac{2|\gamma'(t)|}{1+|\gamma(t)|^2} \, dt.
\end{align}
Geometrically this is exactly the length of $\Pi^{-1}(\gamma(t))$ on the unit sphere, where $\Pi$ is the sterographic projection map from the north pole of
$\mathbb{S}^2$ to the complex plane.

If $f$ is a meromorphic function and $\gamma$ is a curve in the plane,
then the spherical length of the image curve $f\circ \gamma$ is given by
\begin{align}
\label{imagelength}
\int_{\gamma}\frac{2|f'(z)|}{1+|f(z)|^2}\, d|z|.
\end{align}
This leads to the definition of the spherical derivative of $f(z)$ given
by \eqref{sphericalderivative}.
\end{remark}

\begin{remark}
\label{res}
In formula \eqref{complexrepresentation}, if $f$ has a critical point, then $u$ is singular at the point. If $f$ has a pole $z_0$, then since \begin{align*}
    f^\#=\left(\frac{1}{f}\right)^\#,
\end{align*}$z_0$ must be simple otherwise it would be a critical point of $1/f$, and this would make the solution singular. Hence a developing function for a classical solution to \eqref{liouvilleequation} must be locally univalent and has at most simple poles.
\end{remark}

As a consequence of \eqref{complexrepresentation} and Remark \ref{sphericallength}, we give the following observation, which further explains the relation between  Liouville equation \eqref{liouvilleequation} and the spherical derivative
of meromorphic functions.
\begin{proposition}
\label{ob}
Given a locally univalent meromorphic function $f$ with at most simple poles, if we define the length of a curve $\gamma$ as \eqref{imagelength}, then this
induces an conformal metric on $\mathbb{R}^2$ of constant curvature $1$.

Conversely, given a a classical solution $u$ to
\eqref{liouvilleequation}, there exists a meromorphic function $f$
such that for any piecewise smooth curve $\gamma$, the length of $\gamma$
under the metric $e^{2u}\delta$,
is exactly the spherical length of $f\circ \gamma$.
\end{proposition}

The following remark gives a more precise description of developing
functions with respect to solutions to the Liouville equation in $\mathbb{R}^2$.
\begin{remark}
\label{equivliouville}
A developing function $f$ must either be a linear-fractional transform,
or a transcendental meromorphic function.
The former case correspond to solutions radially symmetric about a point.
\end{remark}
\begin{proof}
A rational function cannot be locally univalent unless it is of degree $1$.
\end{proof}


The next two propositions classify radial solutions and one-dimensional solutions to the Liouville equation \eqref{liouvilleequation} as well as the corresponding developing functions.
\begin{proposition}
\label{clradial}
$u$ is a radial solution to \eqref{liouvilleequation}, if and only if the developing function is a linear-fractional transform.
\end{proposition}

\begin{proof}
Without loss of generality, we assume $f(0)=0$. This is achieved by the rotation of the image sphere. Since $f$ is locally univalent, we consider a small circle $|z|=r$, so that $f$ has only one zero and no poles in $|z|<r$. Then, since $f^\#$ is radially symmetric, the image of $|z|=r$ belongs to a small circle on the sphere centered at the north pole, which under the stereographic projection
corresponds to a circle $|w|=s$. Then by the reflection principle, $f$, which maps a circle into a circle,  extends to a linear-fractional transformation.
\end{proof}

\begin{proposition}
\label{class1d}
Let $u$ be a solution to  Liouville equation \eqref{liouvilleequation} with $f$ to be its developing function. Then the following conditions are equivalent
\begin{itemize}
    \item $u$ is one-dimensional.
    \item $u$ takes the form
\begin{align}
    \label{1dform}
u(x,y)=\ln\lambda+\ln \left(\sech\left( \lambda (\omega \cdot (x,y) +b)\right)\right), \quad \mbox{for some $\lambda>0, b \in \mathbb{R}$ and $\omega \in \mathbb{S}^1$.}
\end{align}
\item $f$ takes the form
\begin{align}
\label{1dev}
    f(z)=\frac{pf_1-\bar{q}}{qf_1+\bar{p}},
\end{align}where $p,q \in \mathbb{C}$, $|p|^2+|q|^2=1$, and $f_1=ae^{bz}$ with $a, b \in \mathbb{C}\setminus \{0\}$.
\end{itemize}

\end{proposition}

\begin{proof}
Up to the transformations discussed in Remark \ref{ts}, we may assume that $u$ depends only on the $x$ variable, and we write $u=u(x)$. Multiplying $u'$ on both sides of \eqref{liouvilleequation}, we have
\begin{align}
\label{pp}
    (u')^2+e^{2u}\equiv constant:=\lambda ^2,
\end{align}where $\lambda>0$ is some constant.
Hence $e^u \le \lambda$, and we may introduce another function $v=v(x)$ such that $e^u=\lambda \sech v$. Hence $u'=-v'\tanh v$. Plugging this into \eqref{pp}, we have $v=\pm x+C$, where $C$ is a constant. Hence $u=\ln \lambda+\ln \sech \lambda(\pm x+C)$.  This proves \eqref{1dform}.

Let $\omega=(\omega_1,\omega_2) \in \mathbb{S}^1$, $\lambda>0$ and $b\in \mathbb{R}$. It is easy to check that when $f(z)=e^{b\lambda}e^{\lambda (\omega_1-i\omega_2)z}$, $\ln f^\#$ is exactly \eqref{1dform}, and hence the developing functions of one-dimensional solutions take the form \eqref{1dev}.

Finally, if $f$ is given by \eqref{1dev}, then direct computation implies
that $u$ is one-dimensional.
\end{proof}

\section{Proof of Theorem \ref{e1}, Theorem \ref{upb} and Theorem \ref{answer}}

Before proving Theorem \ref{e1}, let us recall some basic definitions and facts of Nevanlinna theory, for the courtesy of reader.  A good general reference is \cite{H} and Nevanlinna's own book \cite{N}.

Let $f$ be a meromorphic function, and $f^\#$ be its spherical derivative given by \eqref{sphericalderivative}. Consider the function
$$A(r,f)=\frac{1}{4\pi}\int_{|z|\leq r}(f^\#(z))^2dxdy.$$
This is the ratio of the area of the disk $\{ z:|z|\leq r\}$
with respect to the conformal metric with length element $f^\#|dz|$ to the area of the Riemann sphere. So it is interpreted as the
``average covering number'' of the sphere by the restriction of $f$
on $\{ z:|z|\leq r\}$. Then the Nevanlinna characteristic is defined by
$$T(r,f)=\int_0^r A(t,f)\frac{dt}{t}.$$
Nevanlinna characteristic can be considered as a generalization of
degree of a rational function to transcendental meromorphic functions. For every non-constant $f$ we  have $T(r,f)\to\infty,\;
r\to\infty$, and $f$ is rational if and only if
$$\limsup_{r\to\infty}\frac{T(r,f)}{\log r}<\infty,$$
in which case this limit exists and is equal to degree of $f$.

The order of a meromorphic function is defined by
$$\rho=\rho(f):=\limsup_{r\to\infty}\frac{\log T(r,f)}{\log r},$$
and a function of finite order is said to be of {\em normal type} if
$$\limsup_{r\to\infty}r^{-\rho}T(r,f)$$
is finite and non-zero.

From these definitions one easily obtains:
\vspace{.1in}

{\em If $A(r,f)=O(r^p),\; z\to\infty$ for some real $p$, then
$T(r,f)=O(r^{p}).$} Hence $\rho(f) \le p$.
\vspace{.1in}

If moreover $f$ is locally univalent, then we can say much more. Locally univalent meromorphic functions of finite order have been described
by Nevanlinna \cite{Nev1}-\cite{N} as follows. This class of meromorphic functions coincides with ratios of linearly independent
solutions of (\ref{ODE}) with a polynomial $P$, see also \cite[Theorem 16]{Book}. This is the most important property we use to prove Theorem \ref{e1}.

We give a brief sketch of Nevanlinna's proof of this property. It is easy to show that for every meromorphic function, the Schwarzian derivative given by
\begin{align*}
    Sf:=\left(\frac{f^{''}}{f'}\right)'-\frac{1}{2}\left(\frac{f^{''}}{f'}\right)^2,
\end{align*}is a meromorphic function whose
poles are exactly the critical points of $f$ (that is zeros of $f'$ or
multiple poles). Since our function $f$ is locally univalent, its Schwarzian derivative is entire, which we denote by $2P$. Then the assumption that $f$
is of finite order allows an estimate of $P$
which shows that $P$ is in fact a polynomial.
The estimate is performed with the help of the Lemma on
the Logarithmic Derivative, see for example \cite{H}
which is the main technical result of Nevanlinna's theory.

\medskip

Now let us turn to the proof of Theorem \ref{e1}.
\begin{proof}[Proof of Theorem \ref{e1}]
Let $u \in N(k)$ and $f$ be the developing function of $u$.  By Remark \ref{equivliouville}, either $f$ is a linear-fractional transformation, or $f$ is transcendental meromorphic.

For the former case, since up to transformation in Remark \ref{ts},
$u=\ln\left(2/(1+|z|^2)\right)$, and hence $u \in N(-2)$.

For the later case, if in addition $u$ belongs to $N(k)$ for some $k$, then $f^\#(z)=e^{u(z)}=O(|z|^k)$, and hence $f$ has finite order, with order at most $2k+2$. Moreover, by Nevanlinna theory, the Schwarzian derivative of $f$ is a polynomial, which we denote by $P$. Let $\rho(f)$ be the order of $f$, and $d$ be the degree of $P$. Hence
we can write $f=w_1/w_2$, where $w_1$ and $w_2$ are two linearly independent solutions to
\begin{align}
    \label{ODE'}
w''+Pw=0.
\end{align}
The order $\rho(f)=d/2+1$, and $f$ is of normal type.

There is a powerful asymptotic integration theory for \eqref{ODE'}.
The most complete
reference for this theory is Sibuya's book \cite{S},
one can also consult \cite{Eremenko04} and the recent survey \cite{G}.
Let $a$ be the leading coefficient of $P$, and the Stokes
lines for \eqref{ODE'} are defined by
\begin{align*}
	{\mathrm{Im}}\left(az^{d/2+1}\right)=0.
\end{align*}
The asymptotic theory says that the Stokes lines break the complex plane into $d+2$ disjoint open sectors. Moreover, in each of the sector, $u(z)$ tends to $-\infty$, while along Stokes lines, $u$ grows like $(d/2)\ln|z|$. Hence by definition of $N(k)$, if $k \ne -2$, then $d=2k$, and hence $k$ must be a nonnegative half integer.

Therefore, we have shown $u \in N(k)$ if and only if $k=-2$ or $2k$ is a nonnegative integer.

It remains to prove \eqref{asymptotics}. This also follows from the asymptotic theory for the complex linear ODE \eqref{ODE'}. In each sector $S_j$, $f(z)\to a_j$ as $z\to\infty, z\in S_j$,
where $a_j$ are some points in $\bar{C}$. They are called {\em asymptotic values}
and have the property that $a_{j+1}\neq a_j$ and there are at least three distinct $a_j$ unless $d=0$. Since $f$ can be composed with arbitrary rotation of the sphere, we may assume without loss of generality that
all $a_j$ are finite.

The asymptotic behavior is obtained from \cite[Theorem 6.1]{S}.
Assume without loss of generality that the polynomial $P$ in (\ref{ODE})
is monic. Then $S_1=\{ z:|\arg z|<\pi/(d+2)\}$, and we enumerate the sectors
counterclockwise. In each sector $S_j$ equation (\ref{ODE})
has the so-called subdominant
solution which tends to zero exponentially along the rays in this sector.
For the sector $S_1$, the subdominant solution $w_1$ has the
following asymptotic behavior:
\begin{equation}\label{as}
w_1(z)=\exp\left(-cz^\rho+o(z^\rho)\right)
\end{equation}
where $c>0$ and $|z|\to\infty$, uniformly with respect to the $\arg z$ for
$$|\arg z|\leq \frac{3\pi}{d+2}-\epsilon,$$
for every $\epsilon>0$. It is important that the asymptotics hold
in three adjacent
sectors. If $w_2$ is the subdominant solution for $S_2$, then the ratio
$f_0=w_1/w_2$ has the similar asymptotic to (\ref{as}) with $2c$ instead
of $c$ in the sector $$\{z: -\pi/(d+2)+\epsilon<\arg z<3\pi/(d+2)-\epsilon\}$$
which contains both sectors $S_1$ and $S_2$.
Since asymptotics can be differentiated (Sibuya actually states the asymptotic
of derivative in his Theorem 6.1) and arithmetic operations can be
performed on them, we obtain the asymptotic formula
(\ref{asymptotics}) for $f_0^\#$.
Since $f=L\circ f_0$, where $L$ is a linear-fractional transformation,
and linear fractional transformations are bi-Lipschitz with respect to
the spherical metric, we obtain that $f^\#$ obeys the same asymptotic
formula. The argument applies to all pairs of adjacent sectors
instead of $S_1$ and $S_2$.
The statement on the
behavior of $u$ near the boundaries of the sectors is also
obtained from the same asymptotic formula for $f_0$.
\end{proof}

\begin{proof}[Proof of Theorem \ref{upb}]
Let $f$ be the developing function for $u$.

If $u$ is bounded from above, then $u \in N(-2)\cup N(0)$. For the former case, up to transformation in Remark \ref{ts},  $u$ is given by \eqref{cr}.

For the latter case, as discussed in the proof of Theorem \ref{e1}, the Schwarzian derivative of $f$ is a constant, and hence $f$ must be of type $L(e^{\zeta z})$, where $\zeta$ is a nonzero complex number and $L$ is a linear-fractional transformation. The transformation $f(\cdot) \mapsto f(\frac{1}{\zeta}\cdot)$ belongs  to what is discussed in the Remark \ref{ts}, and hence we may assume that
\begin{align*}
    f=\frac{ae^z+b}{ce^z+d}, \quad ad-bc\ne 0.
\end{align*}

Case 1: If $b\ne 0$ and $d=0$, then $f$ is of the form $A+Be^{-z}$. Then there exists $z_0\in \mathbb{C}$ such that $Be^{z_0}$ has norm $1$ and has the same argument as that of $A$. Hence $f(-z+z_0)$ is of the form $e^{i\theta} (|A|+e^z)$, for some $\theta \in \mathbb{R}$. Hence $u(-z+z_0)$ has developing function $|A|+e^z$. Let $|A|=t$, and thus $u(-z+z_0)$ is given by \eqref{u_t}.

Case 2: If $b=0$ and $d \ne 0$, then since $f^\#=(1/f)^\#$, we go back the previous case.

Case 3: If $b \ne 0$ and $d\ne 0$, then there exists $p, q \in \mathbb{C}$ with $|p|^2+|q|^2=1$ such that $bq+d\bar{p}=0$, and hence \begin{align*}
    \tilde{f}:=\frac{pf-\bar{q}}{qf+\bar{p}}
\end{align*}belongs to Case 1.

Therefore, we conclude that when $u \in N(0)$, up to the transformation in Remark \ref{ts}, $u$ is given by \eqref{u_t}.
\end{proof}

\begin{corollary}
Let $u$ be a solution to  Liouville equation \eqref{liouvilleequation}. If $u$ is strictly decreasing along every  ray from the origin, then $u$ is radial.
\end{corollary}
\begin{proof}
By the classification results in Theorem \ref{upb},
radially strictly decreasing function cannot take the form \eqref{u_t},
and thus $u$ must be radial about a point.
Clearly, the point has to be the origin.
\end{proof}

Now from Theorem \ref{upb}, we have the following corollaries.
\begin{corollary}
Let $u$ be a solution to the Liouville equation \eqref{liouvilleequation}. If
\begin{align}
\label{as2}
    \lim_{|z|\rightarrow \infty}u(z)=-\infty,
\end{align}then $u$ is radial about a point.
\end{corollary}
\begin{proof}
By \eqref{as2}, $u$ cannot be of type \eqref{u_t}. Since \eqref{as2} implies that $u$ is bounded from above, by Theorem \ref{upb}, $u$ must be radial about a point.
\end{proof}

\begin{corollary}
Let $u$ be a solution to the Liouville equation \eqref{liouvilleequation}. If $u$ is symmetric about $x$-axis and $y$-axis, and moreover $u_x<0,\, u_y<0$ in the first quadrant, then $u$ is radial.
\end{corollary}
\begin{proof}
By hypothesis, $u$ is bounded from above. Then by Theorem \ref{upb},
$u$ must be either of the form (\ref{cr}) or
of the form (\ref{u_t}), up to the transformation
in Remark 1.5. A direct check shows that (\ref{u_t}) is excluded,
so $u$ must be radial about a point, and the point must be the origin.
\end{proof}
\begin{corollary}
\label{2d}
Let $u$ be a solution to  Liouville equation \eqref{liouvilleequation}. If $u$ is bounded from above, then $u=\ln(\sech y)$, up to transformation
in Remark \ref{ts}.
\end{corollary}
\begin{proof}
Clearly radial solutions are not concave. For the family of solutions given by \eqref{u_t}, only when $t=0$ the solution is concave. Hence up to the transformations discussed in Remark \ref{ts}, $u=\ln\sech(y)$.
\end{proof}


\begin{corollary}
If a solution $u$ is bounded from above, then under the metric $e^{2u}\delta$, the range of diameter of $\mathbb{R}^2$ is exactly $[\pi, 2\pi)$. The diameter equals to $\pi$ if and only if either $u$ is radial about a point, or $u$ is one-dimensional.
\end{corollary}
\begin{proof}
By Theorem \ref{upb}, up to the transformation discussed in Remark \ref{ts}, $u$ is either given by \eqref{cr} or \eqref{u_t},  where the diameter equals to $\pi$ and $\pi+2\tan^{-1}t$ respectively, see \cite[Proposition 1.2]{GL}.

The first fact is clear, and for the courtesy of the reader, we prove the latter fact. The proof is simpler than that in \cite{GL}.

We let $L_{\pm}=\{(x,\pm\pi): x\in \mathbb{R}\}$ and $L_0$ be the real axis. We also let $O$ be the origin, $N$ be the north pole on $S^2$, and $A$ be the point $(\frac{2t}{1+t^2},0,\frac{t^2-1}{1+t^2})$.
Note that a solution given by \eqref{u_t} has a developing function $f=t+e^z$. Hence in the Riemann sphere, $f(L_{\pm})$ is exactly the geodesic from $A$ to $N$ along the south pole. Hence we can choose $P \in L_+$ and $Q\in L_-$ such that in the Riemann sphere, both $f(P)$ and $f(Q)$ lie in the middle of the circular arc from $A$ to $N$ along the south pole. Since any curve from $P$ to $Q$ must intersect $L_0$, and $f(L_0)$ is the minimizing geodesic from $A$ to $N$ along the northern sphere, we have that the spherical length of $f\circ \gamma$ must be bounded below by $2\pi-\alpha$, where $\alpha$ is the angle between $OA$ and $ON$. Note that
\begin{align*}
    \tan \alpha=\frac{2t}{t^2-1},
\end{align*}
we have that $t=\tan(\frac{\pi-\alpha}{2})$, and hence $\alpha=\pi-2\tan^{-1}t$. Hence the spherical length of $f\circ \gamma$ is bounded below by $\pi+2\tan^{-1}t$.

To prove the upper bound, note that
\begin{align*}
    \int_{-\infty}^{\infty}e^{u(x,y)}dx=&\frac{2}{\sqrt{1+t^2\sin^2 y}}\tan^{-1}(\frac{t\cos y+e^x}{\sqrt{1+t^2 \sin^2 y}})\Big|_{-\infty}^{\infty}\\
    =&\frac{2}{\sqrt{1+t^2\sin^2 y}}\left(\frac{\pi}{2}-\tan^{-1}(\frac{t\cos y}{\sqrt{1+t^2 \sin^2 y}})\right).
\end{align*}
Hence \begin{align}
    \label{sup'}
\sup_{y \in \mathbb{R}}\int_{-\infty}^{\infty}e^{u(x,y)}dx=\pi+2\tan^{-1}(t).
\end{align}
Also, it is easy to see that for fixed $-\infty<T_1<T_2<\infty$.
\begin{align}
\label{w0}
   \lim_{x\rightarrow \pm \infty} \int_{T_1}^{T_2}e^{u(x,y)}\, dy=0.
\end{align}Hence by \eqref{sup'} and \eqref{w0}, we have that $diam_g(\mathbb{R}^2) \le \pi+2\tan^{-1}(t)$.

Hence if $u$ is given by \eqref{u_t}, the diameter of $\mathbb{R}^2$ is exactly $\pi+2\tan^{-1}(t)$, and for such family, only when $t=0$, the diameter is equal to $\pi$.

Note that under the metric $e^{2u}\delta$,
the diameter of $\mathbb{R}^2$ does not change
if we perform the transformations discussed in Remark \ref{ts},
hence we proved the corollary.
\end{proof}

\begin{proof}[Proof of Theorem \ref{answer}]
The  theorem  is a combination of the above four corollaries.
\end{proof}

\section{Proof of Theorem \ref{main}}
The following lemma is crucial in proving Theorem \ref{main} in any dimensions.
\begin{lemma}
\label{ne}
Let $u$ be a solution to \eqref{nliouvilleequation}, where $n \ge 3$. If further more $u$ has an upper bound and
\begin{align}
    \label{fv}
\int_{\mathbb{R}^n}e^{2u}\, dx<\infty,
\end{align}then such solution $u$ does not exist.
\end{lemma}

\begin{proof}
If \eqref{fv} is true, then we can define
\begin{align}
    \label{w}
w(x):=-\frac{1}{n(n-2)\omega_n}\int_{\mathbb{R}^n}    |x-y|^{2-n}e^{2u(y)}\, dy,
\end{align}where $\omega_n$ is the volume of the unit ball in $\mathbb{R}^n$. Such $w$ is well-defined. Also, for any $x$, since $u$ has an upper bound, the value of the above integral over $B_1(x)$ is uniformly bounded. Over the complement of $B_1(x)$, the integral is also uniformly bounded, due to the condition \eqref{fv}. Therefore, $w(x) \in L^{\infty}(\mathbb{R}^n)$.

Note that $\Delta w(x)=e^{2u(x)}$, and hence
\begin{align*}
    \Delta (u+w)=0\quad \mbox{in $\mathbb{R}^n$.}
\end{align*}Since both $u$ and $w$ has an upper bound, we conclude by Liouville Theorem that $u+v$ is a constant in $\mathbb{R}^n$. Since $w \in L^{\infty}(\mathbb{R}^n)$, $u \in L^{\infty}(\mathbb{R}^n)$. This however contradicts to \eqref{fv}, since it would require that $\liminf_{|x|\rightarrow \infty}u(x)=-\infty$.
\end{proof}

\begin{remark}
Lemma \ref{ne} does not hold in two dimensions. The reason why it is true in higher dimensions is because the fundamental solution to $-\Delta$ decays to zero near infinity, while in two dimensions this is not true.
\end{remark}

Now we can prove Theorem \ref{main}.

\begin{proof}[Proof of Theorem \ref{main}]
Without loss of generality we assume that $u$ has a local maximum at $0$, and thus $u$ achieves its supremum at $0$. Suppose that $u$ is strictly concave, then for $z \in \mathbb{R}^n$, we have
\begin{align*}
    \sup_{|z|=1}\left(u(z)-u(0)\right)\le -a,\quad \mbox{for some $a>0$.}
\end{align*}
Again by concavity,
\begin{align*}
    \sup_{|z|\ge 1}\left(\frac{u(z)-u(0)}{|z|}\right)\le -a,\quad \mbox{for some $a>0$.}
\end{align*}
Hence for $|z|\ge 1$,
\begin{align*}
    \frac{u(z)}{|z|}=\frac{u(z)-u(0)}{|z|}+\frac{u(0)}{|z|}\le -a+\frac{u(0)}{|z|}.
\end{align*}
Hence
\begin{align}
\label{lineardecay}
    u(z)\le -a|z|+u(0).
\end{align}
Therefore, $\int_{\mathbb{R}^n}e^{2u(y)}\, dy<\infty$. When $n\ge 3$, this is impossible due to Lemma \ref{ne}. When $n=2$, $u$ is radial, but a radial solution cannot be concave, and thus we also obtain a contradiction.

Hence $u$ is not strictly concave, and thus by the constant rank theorem (see \cite[Theorem 2]{KL87}), $D^2 u$ has constant rank $r$, with $r<n$. Hence $u$ is a constant in $n-r$ coordinates. If $r>1$, then we may assume that $u$ is a constant in $x_{n-r+1}, \cdots, x_n$ variables. Define a function $U(x'):=u(x', x_{n-r+1}, \cdots, x_n)$, where $x' \in \mathbb{R}^{r}$. Hence $U$ is a strictly concave entire solution to the Liouville equation in $r$ dimensions. However, applying the previous argument, $U$ cannot be strictly concave, unless $r=1$. Hence we have shown that $D^2 u$ has constant rank $1$, and hence level sets of $u$ are flat.
\end{proof}

\section{Further remarks}

Questions \ref{q2}-\ref{q3} remain open in their full generality. On the two
questions, we propose the following conjectures.

\begin{conjecture}
Let $u$ be a solution to  Liouville equation \eqref{liouvilleequation}. If under the metric $e^{2u}\delta$ where $\delta$ is the standard Euclidean metric, $\mathbb{R}^2$ has diameter $\pi$, then $u$ must be either radial about a point, or one-dimensional.
\end{conjecture}

\begin{conjecture}
Let $u$ be a solution to  Liouville equation \eqref{liouvilleequation}. If $u$ is concave, then $u$ is one-dimensional.
\end{conjecture}

We remark that motivated by Question \ref{q2}, stronger constant rank theorems  have  been proved in \cite{LX2021}, which are stated as follows.

\begin{theorem}
\label{sigma2theorem}
Let $u$ be a convex solution to the semilinear elliptic equation $\Delta u=G(u)$ in  $\mathbb{R}^n$, with $G>0, G'<0$ and $GG^{''}<2(G')^2$. If $\sigma_2(D^2 u)$ has a local minimum in the interior, then level sets of $u$ are hyperplanes.
\end{theorem}

\begin{theorem}
\label{sigman}
Let $u$ be a convex solution to $\Delta u=G(u)$ in  $\mathbb{R}^n$, with $G>0, G'<0$ and $GG^{''}\le \frac{n}{n-1}(G')^2$. Then if $det(D^2 u)$ has a local minimum, then $det(D^2 u) \equiv 0$.
\end{theorem}

In particular, if one can prove that   $det(D^2 u)$  attains  a local minimum for any concave solution $u$ to Liouville equation \eqref{nliouvilleequation} in arbitrary dimensions, then by Theorem \ref{sigman} and dimension reduction, $u$ must be  one-dimensional.  However,  whether the assumption holds  still remains  an open question.


\begin{thebibliography}{10}
\bibitem{AC20} L. Ambrosio and X. Cabr\'e, Entire solutions of semilinear elliptic equations in $\mathbb{R}^3$ and a conjecture of De Giorgi, Journal Amer. Math. Soc. 13 (2000), 725--739.



\bibitem{BG} B. Bian and P. Guan, A microscopic convexity principle for nonlinear partial differential equations, Invent. Math. 177 (2009), 307--335.

\bibitem{BGMX} B. Bian, P. Guan, X. Ma and L. Xu, A constant rank theorem for quasiconcave  solutions of fully nonlinear partial differential equations, Indiana Univ. Math. J. 60 (2011), no.1, 101--120.

\bibitem{CGM} L. Caffarelli, P. Guan and X. Ma, A constant rank theorem for solutions of fully nonlinear elliptic equations, Comm. Pure Appl. Math. 60 (2007), 1769--1791.

\bibitem{CF} L. Caffarelli and A. Friedman, Convexity of solutions of some semilinear elliptic equations,
Duke Math. J. 52 (1985), 431--455.

\bibitem{CLMP} E. Caglioti, P. Lions, C. Marchioro and M. Pulvirenti, A special class of stationary
flows for two dimensional Euler equations: a statistical mechanics description. II.,
Comm. Math. Phys. 174 (1995), 229--260.

\bibitem{CK94} S. Chanillo and M. Kiessling, Rotational symmetry of solutions of some nonlinear problems in statistical mechanics and in geometry, Comm. Math. Phys. 160 (1994), 217--238.


\bibitem{CL} W. Chen and C. Li, Classification of solutions of some nonlinear elliptic equations, Duke Math. J. 63 (1991), no.3,  615--622.

\bibitem{CW94} K. Chou and T. Wan, Asymptotic radial symmetry of solution to $\Delta u+e^u=0$ in a punctured disk, Pacific J. Math. 163 (1994), no.2, 269--276.

\bibitem{CH66} J. Clunie and W. Hayman, The spherical derivative of integral and meromorphic functons, Comment Math. Helv. 40 (1966), 117--148.

\bibitem{DKW} M. del Pino, M. Kowalczyk and J. Wei, On De Giorgi conjecture in dimension $N \ge 9$, Ann. of Math. (2) 174 (2011), no.3, 1485--1569.

\bibitem{DF} E. Dancer and A. Farina, On the classification of solutions of $-\Delta u=e^u$ on $\mathbb{R}^n$: stability outside a compact set and applications, Proc. Amer. Math. Soc. 137 (2009), no.4, 1333--1338.

\bibitem{Book} A. Eremenko, Entire and meromorphic solutions of ordinary differential equations, Chapter 6 in the book: Complex Analysis I, Encyclopaedia of Mathematical Sciences, vol. 85; Springer, NY, (1997), 141--153.

\bibitem{Eremenko04} A. Eremenko, A Toda lattice in dimension 2 and Nevanlinna theory, J. Math. Phys., Anal. Geom. 31 (2007), 39--46.

\bibitem{Farina} A. Farina, Stable solutions of $-\Delta u = e^u$ on $\mathbb{R}^N$, C. R. Math. Acad. Sci. Paris 345 (2007), no.2, 63--66.

\bibitem{GG98} N. Ghoussoub and C. Gui, On a conjecture of de Giorgi and some related problems,
Math. Ann. 311 (1998), 481--491.

\bibitem{GG03} N. Ghoussoub and C. Gui, On De Giorgi's conjecture in dimensions
4 and 5, Ann. of Math. (2) 157 (2003), no. 1, 313--334.

\bibitem{GLM} P. Guan, C. Lin and X. Ma, The Christoffel-Minkowski problem II: Weingarten curvature
equations, Chin. Ann. Math., Ser. B, 27 (2006), 595--614.

\bibitem{GM} P. Guan and X. Ma, The Christoffel--Minkowski problem I: Convexity of solutions of a
Hessian equations, Invent. Math. 151 (2003), 553--577.

\bibitem{GMZ} P. Guan, X. Ma and F. Zhou, The Christoffel-Minkowski problem III: existence and convexity of admissible solutions, Comm. Pure Appl. Math. 59 (2006), 1352--1376.

\bibitem{GX} P. Guan and L. Xu, Convexity estimates for level sets of quasiconcave solutions to fully nonlinear elliptic equations, J. Reine Angew. Math. 680 (2013), 41--67.


\bibitem{GL} C. Gui and Q. Li, Some geometric inequalities related to Liouville equations, Preprint.

\bibitem{G} G. Gundersen, J. Heittokangas and A. Zemirni,
Asymptotic integration theory for $f''+P(z)f=0$, Expo. Math. 40 (2022), no. 1, 94--126.

\bibitem{H} W. Hayman, Meromorphic functions,
Oxford, Clarendon Press, 1964.



\bibitem{KL87} N. Korevaar and J. Lewis, Convex solutions of certain elliptic equations have constant rank Hessians, Arch. Rational Mech. Anal. 97 (1987), 19--32.



singularity, Comment. Math. Helv. 33 (1959), 196--205 .


\bibitem{LX2021} Q. Li and L. Xu, A stronger constant rank theorem, Preprint.

\bibitem{Liouville} J. Liouville, Sur l'\'equation aux differences partielles $\partial^2\log\lambda/\partial u\partial v\pm\lambda^2/2a^2=0$, J. Math.  18 (1853),
71--72.

\bibitem{MX08} X. Ma and L. Xu, The convexity of solutions of a class of Hessian equation in bounded convex domain in $\mathbb{R}^3$, J. Funct. Anal. 255 (2008), 1713--1723.

\bibitem{Nev1} R. Nevanlinna, \"Uber Riemannsche Fl\"achen mit endlich
vielen Windungspunkten, Acta Math. 58(1) (1932), 295--373.

\bibitem{N} R. Nevanlinna, Analytic functions, Springer, NY, 1970.


\bibitem{Savin} O. Savin, Regularity of flat level sets in phase transitions, Ann. of Math. (2) 169 (2009), no.1,  41--78.

\bibitem{S} Y. Sibuya, Global theory of a second order linear ordinary differential equation with a polynomial coefficient,
North Holland, NY, 1975.

\bibitem{T} G. Tarantello, Self-Dual Gauge Field Vortices: An Analytical Approach, PNLDE 72, Birkh\"auser Boston, Inc., Boston, MA, 2007.

\bibitem{Wang2017} K. Wang, A new proof of Savin's theorem on Allen-Cahn equations, J. Eur. Math. Soc. 19 (2017), 2997--3051.

\bibitem{WW20} K. Wang and J. Wei, Finite Morse index implies finite ends, Comm. Pure Appl. Math. 72 (2019), no.5, 1044--1119.

\bibitem{Y} Y. Yang, Solitons in Field Theory and Nonlinear Analysis, Springer Monographs in
Mathematics, Springer, New York, 2001.



\end{thebibliography}
\end{document}